\documentclass{amsart}

\usepackage{amsthm, amsfonts, amssymb, amsmath, latexsym, enumerate, times}
\usepackage[latin1]{inputenc}
\usepackage{mathrsfs}
\usepackage{mathtools}
\usepackage{letltxmacro}
\LetLtxMacro\orgvdots\vdots
\LetLtxMacro\orgddots\ddots

\makeatletter
\DeclareRobustCommand\vdots{%
	\mathpalette\@vdots{}%
}
\newcommand*{\@vdots}[2]{%
	\sbox0{$#1\cdotp\cdotp\cdotp\m@th$}%
	\sbox2{$#1.\m@th$}%
	\vbox{%
		\dimen@=\wd0 %
		\advance\dimen@ -3\ht2 %
		\kern.5\dimen@
		\dimen@=\wd2 %
		\advance\dimen@ -\ht2 %
		\dimen2=\wd0 %
		\advance\dimen2 -\dimen@
		\vbox to \dimen2{%
			\offinterlineskip
			\copy2 \vfill\copy2 \vfill\copy2 %
		}%
	}%
}
\DeclareRobustCommand\ddots{%
	\mathinner{%
		\mathpalette\@ddots{}%
		\mkern\thinmuskip
	}%
}
\newcommand*{\@ddots}[2]{%
	\sbox0{$#1\cdotp\cdotp\cdotp\m@th$}%
	\sbox2{$#1.\m@th$}%
	\vbox{%
		\dimen@=\wd0 %
		\advance\dimen@ -3\ht2 %
		\kern.5\dimen@
		\dimen@=\wd2 %
		\advance\dimen@ -\ht2 %
		\dimen2=\wd0 %
		\advance\dimen2 -\dimen@
		\vbox to \dimen2{%
			\offinterlineskip
			\hbox{$#1\mathpunct{.}\m@th$}%
			\vfill
			\hbox{$#1\mathpunct{\kern\wd2}\mathpunct{.}\m@th$}%
			\vfill
			\hbox{$#1\mathpunct{\kern\wd2}\mathpunct{\kern\wd2}\mathpunct{.}\m@th$}%
		}%
	}%
}
\makeatother

\usepackage{array}
\usepackage{comment}
\usepackage{paralist}
\usepackage{xcolor}

\newtheorem{theorem}{Theorem}
\newtheorem{lemma}[theorem]{Lemma}

\newtheorem{corollary}[theorem]{Corollary}
\newtheorem{proposition}[theorem]{Proposition}

\theoremstyle{definition}

\newtheorem{remark}[theorem]{Remark}
\newtheorem{problem}[theorem]{Problem}

\newcommand{\calI}{{\mathcal I}}
\newcommand{\calJ}{{\mathcal J}}

\newcommand{\calP}{{\mathcal P}}

\newcommand{\calO}{{\mathcal O}}
\newcommand{\calX}{{\mathcal X}}

\newcommand{\bbP}{{\mathbb P}}

\newcommand{\bbC}{{\mathbb C}}
\newcommand{\bbK}{{\mathbb K}}

\def\geq{\geqslant}
\def\leq{\leqslant}
\def\le{\leqslant}
\def\ge{\geqslant}

\begin{document}
 
\title[Two letters by Guido Castelnuovo]{Two letters by Guido Castelnuovo}

\author{Ciro Ciliberto}
\address{Dipartimento di Matematica, Universit\`a di Roma Tor Vergata, Via O. Raimondo
 00173 Roma, Italia}
\email{cilibert@axp.mat.uniroma2.it}

\author{Claudio Fontanari}
\address{Dipartimento di Matematica, Universit\`a degli Studi di Trento, Via Sommarive 14, 38123 Povo, Trento}
\email{claudio.fontanari@unitn.it}
 
 

\begin{abstract} In this expository paper we transcribe two letters by Guido Castelnuovo, one to Francesco Severi, the other to Beniamino Segre, and explain the contents of both, which basically focus on the quest for an algebraic proof of the equality between the analytic and the arithmetic irregularity and of the closedness of regular 1--forms on a complex, projective, algebraic surface. Such an algebraic proof has been found only in the 1980's by Deligne and Illusie. 
\end{abstract}

\maketitle 

\section*{Introduction} 
As it is well-known, the treatise by Federigo Enriques epitomizing the celebrated classification of algebraic surfaces by the Italian school of algebraic geometry has been published posthumously in 1949, a few years after the sudden death of the author in 1946. As pointed out by Guido Castelnuovo in the preface (see \cite{E}), 

\begin{quote} 
(...) dove il terreno \`e meno solido l'Autore mette sull'avviso lo studioso. Di questi punti ancora fluidi quello che presenta la difficolt\`a pi\`u ardua ed il maggiore interesse \emph{\`e la teoria} dei sistemi continui di curve algebriche (...) che esistono sopra ogni superficie irregolare. (...) tutti i tentativi compiuti (...) per dimostrarla mediante considerazioni algebrico-geometriche si sono urtati contro difficolt\`a sinora insuperate. (...) l'Autore d\`a anche suggerimenti sopra una via da tentare per giungere alla meta. Debbo confessare che non vedo come quella via possa tradursi in un procedimento irreprensibile.

\medskip

\noindent
(...) where the ground is less solid the Author warns the reader. Among these still unsteady points the most difficult and interesting one \emph{is the theory} of the continuous systems of algebraic curves (...) existing on any irregular surface. (...) all attempts (...) towards an algebro-geometric approach have been frustrated by still insurmountable difficulties. (...) the Author provides some hints about a strategy to reach the goal. I should confess I cannot see how that strategy may be translated into a fully rigorous argument.

\end{quote} 

The two letters by Guido Castelnuovo that we  transcribe and translate into English in \S \ref{sec:letters}, the first one addressed to Francesco Severi and dated 1947, the second one addressed to Beniamino Segre and dated 1950, provide first hand witness of Castelnuovo's attempts to a purely algebro--geometric understanding of irregular surfaces.

In \S \ref{sec:reg} we explain the background of Castelnuovo's letters, using  modern terminology. In particular we  
explain a classical method, very familiar to Castelnuovo and due to Picard and Severi, of constructing regular 1--forms on a surface. As explained in \S \ref {ssec:closed}, one of the crucial points of Castelnuovo's approach is the attempt of proving the closedness of global regular 1-forms, a fact that today we are aware to strongly rely on the characteristic zero assumption. Indeed, Castelnuovo's remarks in his letters turn out to be quite inconclusive and sometimes even unprecise, as we discuss in Section \ref{sec:comments}, that is devoted to explaining most of the issues raised by Castelnuovo in his two letters.  The algebro-geometric proof of the results that Castelnuovo seeked (i.e., closedness of global regular 1-forms and equality of different definitions of irregularity,  in characteristic zero)  is now available, it is due to work by Deligne and Illusie in the 1980's  and turns out to be completely out of reach of Castelnuovo's classical tools, since (somehow paradoxically) it involves a tricky reduction to the case of positive characteristic.  Section \ref{sec:illusie} is devoted to give a brief account on how these algebraic proofs can be obtained using modern tools. 

We stress that the present note does not contain any original result, but in our opinion the contents of Castelnuovo's letters are worthy of 
careful consideration from both an historical and a mathematical viewpoint. This paper is addressed to readers who are well aware of rather advanced concepts in algebraic geometry, so we do not dwell on explaining standard technical details when they occur.  \medskip

{\bf Acknowledgements:} The authors are members of GNSAGA of the Istituto Nazionale di Alta Matematica ``F. Severi". This research project was partially supported by PRIN 2017 ``Moduli Theory and Birational Classification''.

\section{The letters}
\label{sec:letters}

In this section we transcribe two letters of Castelnuovo, the first one of November 26, 1947 to Francesco Severi, the second one of January 15, 1950 to Beniamino Segre. The first letter belongs to the ``Fondo Guido Castelnuovo'' of the Accademia Nazionale dei Lincei, that has been edited by Paola Gario, has been digitalized and can be found on the web page\medskip

\centerline{http://operedigitali.lincei.it/Castelnuovo/Lettere\_E\_Quaderni/menuL.htm}\medskip

The second letter comes from the collection of documents of Beniamino Segre kept at the University of Caltech. 

\subsection{Guido Castelnuovo to Francesco Severi} 

\begin{flushright} 
Roma, 26 novembre 1947
\end{flushright}

Caro Severi, 

aderendo al tuo desiderio ti comunico alcuni risultati sulle superficie irregolari; 
parecchi si ottengono senza difficolt\`a e possono servire come esercizio per i tuoi 
discepoli.

Lo scopo remoto ed ambizioso che mi proponevo era di costruire una teoria delle dette 
superficie indipendente dalla nozione di sistema continuo di curve, teoria in cui si 
ritrovassero il teorema sul numero $(p_g-p_a)$ dei differenziali totali indipendenti 
di prima specie, il teorema di Hodge, ecc.. Il programma \`e appena iniziato; ma si 
deve raggiungere la meta, a meno che la teoria delle superficie irregolari non riservi 
delle sorprese che non saprei nemmeno immaginare. 

\vspace{0.5cm}
1.

Indico con $\vert C \vert$ un sistema regolare di grado $n$ e genere $\pi$; in molti casi
occorre supporre che $\vert C \vert$ sia \emph{abbastanza ampio}, contenga entro di 
s\`e il sistema cano\-nico $\vert K \vert$ od anche un suo multiplo; ricercando caratteri 
invarianti, tutto ci\`o non ha importanza. Indico con $\chi$ il tuo invariante $q'$, 
cio\`e il numero delle curve indipendenti di $\vert 2C + K \vert$ che passano per il 
gruppo jacobiano $G_\delta$ di un fascio $\vert C \vert$ e in conseguenza per il gruppo 
base $G_n$ del fascio. Indico con $G_k$ il gruppo dei $k$ punti cuspidali di una superficie, 
d'ordine $n$, a singolarit\`a ordinarie, le cui sezioni piane appartengano al sistema 
$\vert C \vert$. 

Ecco un significato di $\chi$ che si raggiunge subito:

\vspace{0.3cm}
\noindent
1) \emph{E' $\chi$ la sovrabbondanza del sistema $\vert 4C + 2K \vert$ rispetto 
al gruppo dei punti cuspidali $G_k$} (cio\`e $G_k$ impone $k -\chi$ condizioni al 
detto sistema).

Invece $G_k$ presenta condizioni indipendenti ai sistemi $\vert mC + K \vert$,
$\vert mC + 2K \vert$, \ldots, per $m \ge 5$.

Per $m=4$ vi \`e un risultato di Enriques ottenuto indirettamente attraverso 
il computo dei moduli di una superficie, risultato che converrebbe dimostrare 
direttamente; lo ricordo perch\'e interviene tra poco: "La sovrabbondanza 
del sistema $\vert 4C + K \vert$ rispetto al gruppo $G_k$ dei punti cuspidali 
\`e un invariante", che indicher\`o con $Q'$ e di cui sotto dar\`o l'espressione. 

\vspace{0.3cm}
\noindent
2) \emph{La serie completa $g_k$ determinata dal gruppo $G_k$ sopra una curva 
di $\vert 4C + K \vert$ passante per esso ha la dimensione $\chi$}.

\vspace{0.3cm}
\noindent
3) \emph{La serie completa $g_k$ determinata dal gruppo $G_k$ sopra una curva 
di $\vert 3C + K \vert$ passante per esso (ad es.: sulla $f=f'_x=0$) ha la dimensione
$2n -\pi+2p_g+p_a-(I+4)+\theta$ dove $0 \le \theta \le p_g-p_a$} (si suppone 
$\vert C \vert$ abbastanza ampio). E' $\theta$ un invariante?

\vspace{0.5cm}
2.

Il procedimento che ti ha condotto a stabilire l'invarianza di $q'=\chi$ fa vedere 
subito che:

\vspace{0.3cm}
\noindent 
4) \emph{E' invariante il numero delle curve linearmente indipendenti di 
$\vert 2C + 2K \vert$ che passano per il gruppo jacobiano $G_\delta$ di un fascio 
$\vert C \vert$ ed anche per il gruppo base $G_n$}; indicher\`o questo invariante 
con $Q$.

Si vede poi (se \`e esatto il risultato di Enriques sopra citato) che: 

\vspace{0.3cm}
\noindent 
5) \emph{E' pure invariante il numero delle curve linearmente indipendenti 
di $\vert 2C + 2K \vert$ che passano per il gruppo $G_\delta$ senza esser costrette 
a passare per $G_n$}; questo nuovo invariante uguaglia l'invariante di Enriques $Q'$. 

E' quindi invariante il numero delle condizioni che una curva di $\vert 2C + 2K \vert$
passante per il gruppo jacobiano $G_\delta$ di un fascio $\vert C \vert$ deve soddisfare 
per contenere il gruppo base. Si dimostra che questo invariante soddisfa alla diseguaglianza
$Q'-Q \le p_g$. 

Quanto alle espressioni di $Q$ e $Q'$ posso dir questo. 

Se le $\infty^{Q-1}$ curve di $\vert 2C + 2K \vert$ passanti per $G_\delta + G_n$ segano 
sopra una curva di $\vert 2C + K \vert$ passante per lo stesso gruppo una serie 
\emph{completa} (residua di $G_\delta + G_n$ rispetto alla serie canonica) allora: 
$$
Q-1 = p_a + p_g + p^{(1)} - (I-4) + \omega
$$
dove $\omega$ ($\le p_g-p_a$) \`e un nuovo invariante che ha un significato molto semplice: 
$I+4-\omega-1$ \`e il numero delle condizioni che un gruppo $G_{I+4}$ della tua serie 
d'equivalenza (in senso stretto) presenta alle curve \emph{bicanoniche} costrette 
a contenerlo. 

Se la serie lineare nominata non \`e completa, dall'espressione di $Q-1$ va tolta la 
deficienza $\le p_g-p_a$ della serie stessa. 

Nello stesso ordine d'idee ti comunico ancora questo risultato: 

\vspace{0.3cm}
\noindent 
6) \emph{La sovrabbondanza del sistema $\vert 3C + K \vert$ rispetto al gruppo 
jacobiano $G_\delta$ di un fascio $\vert C \vert$ \`e un invariante e vale precisamente 
$2p_g$} (se il sistema completo $\vert C \vert$ cui il fascio appartiene \`e abbastanza 
ampio).

\vspace{0.5cm}  
3.

Altre questioni. Come sai il teorema fondamentale da dimostrare \`e questo: una curva 
$\Gamma$ di $\vert 2C + K \vert$ passante per il gruppo jacobiano $G_\delta$ e il gruppo 
base $G_n$ di un fascio $\vert C \vert$ sega sopra la curva generica $C$ del fascio 
(fuori di $G_n$) un gruppo canonico \emph{che non appartiene ad una curva aggiunta 
$C'$}. Ho cercato di trasformare la condizione in altre equivalenti. Tale \`e ad esempio 
questa: il gruppo $G_n$ su quella curva deve presentare condizioni indipendenti alla 
serie caratteristica di $\Gamma$ \emph{resa completa}. Alla serie caratteristica 
in senso stretto, $G_n$ presenta solo $n-1$ condizioni. 

Altra forma: Scriviamo la identit\`a di Picard in coordinate omogenee: 
$$
Xf'_x+Yf'_y+Zf'_z+Tf'_t=0
$$
dove $X=0, \ldots$ sono superficie d'ordine $n-3$. Occorre aggiungere 
la condizione (non detta esplicitamente) che le superficie $yX-xY=0, 
\ldots, tZ-zT=0$ siano aggiunte alla $f=0$ d'ordine $n$. Segue che le 
$X=0, \ldots, T=0$ passano semplicemente per i $t$ punti tripli di $f$ 
ed hanno inoltre in comune un gruppo di $(n-4)d-3t$ punti sulla curva 
doppia d'ordine $d$ di $f$; esse segano inoltre rispettivamente i piani 
$x=0, \ldots, t=0$ in curve aggiunte d'ordine $n-3$. Da ci\`o segue che 
quel gruppo di punti della curva doppia appartiene alla serie segata su 
questa dalle superficie d'ordine $n-4$ passanti semplicemente per i $t$ 
punti tripli, purch'e questa serie venga resa completa, mentre essa ha 
la deficienza $p_g-p_a$ per la definizione stessa di irregolarit\`a. 
Orbene il teorema fondamentale equivale al seguente: \emph{Per quel 
gruppo di $(n-4)d-3t$ punti della curva doppia non passa nessuna superficie 
d'ordine $n-4$ che contenga i $t$ punti tripli di $f$}. Questo enunciato 
si traduce in questo altro, molto elegante dal punto di vista analitico: 
\emph{Non \`e possibile soddisfare una identit\`a del tipo:
$$
\overline{X}f'_x+\overline{Y}f'_y+\overline{Z}f'_Z+\overline{T}f'_t 
\equiv Qf
$$
ove $\overline{X}=0, \ldots, \overline{T}=0$ sono superficie 
aggiunte d'ordine $n-3$ e $Q=0$ una superficie d'ordine $n-4$ se non nel 
caso banale $\overline{X}= \frac{1}{n}xQ, \ldots, \overline{T}= \frac{1}{n}tQ$}.
  
Ritornando all'identit\`a di Picard scritta sopra, ti consiglio di far studiare 
da qualche discepolo la omografia tra il sistema di superficie non aggiunte di 
ordine $n-3$ $\lambda X + \mu Y + \nu Z + \rho T = 0$ e il sistema di piani 
$\lambda x + \mu y + \nu z + \rho t = 0$, ognuno dei quali taglia la superficie 
corrispondente in una curva aggiunta. Nel caso delle rigate irrazionali dei 
primi ordini si trovano propriet\`a elegantissime.  

\vspace{0.5cm}
4.

Finalmente alcune osservazioni che ti potranno servire se esporrai in lezione la tua Nota 
sugli integrali semiesatti. 

Tu dimostri che ad ogni curva di $\vert 2C + K \vert$ passante per il gruppo jacobiano 
e per il gruppo base di un fascio $\vert C \vert$ (\emph{curva covariante del fascio}, 
come io la chiamo) si pu\`o associare una determinata curva covariante di ogni altro 
fascio $\vert D \vert$. Due curve associate segano sullo stesso gruppo di punti la 
curva di contatto di due fasci. Esse inoltre si segano in un gruppo $G_{I+4}$ della 
tua serie. Si vede facilmente che questo gruppo \`e comune a tutta la famiglia di 
curve covarianti associate relative agli infiniti fasci esistenti sulla superficie. 
Ogni curva $C$ di $f=0$ \`e segata dalla curva covariante della famiglia in un gruppo 
canonico che dir\`o \emph{gruppo traccia}.

Preso un punto $P$ della superficie, esistono infinite curve per $P$ per le quali  
$P$ appartiene al gruppo traccia. \emph{Tutte queste curve si toccano in $P$}. 
Vuol dire che ad ogni punto $P$ di $f=0$ \`e collegata una direzione tangente, 
o un elemento lineare uscente da $P$ (indeterminato solo se $P$ appartiene al 
gruppo $G_{I+4}$). Connettendo tutti questi elementi si viene a ricoprire la 
superficie con un fascio di curve (trascendenti) che risultano esser le curve 
integrali dell'equazione $Bdx-Ady=0$, dove $A=0$ e $B=0$ sono due superficie 
aggiunte d'ordine $n-2$ secanti su $f$ le curve covarianti dei fasci $x =$ 
cost., $y =$ cost. Resterebbe naturalmente da far vedere che $1/f'_z$ \`e 
fattore integrante dell'espressione differenziale. 

Al variare di $P$ su $f$ quella tangente in $P$ descrive una congruenza algebrica 
di classe $2 \pi - 2$ e di ordine $k - \nu = 6 \pi - 6 + p^{(1)} - 1 - (I+4)$ 
ove $k$ \`e  il numero di punti cuspidali e $\nu$ \`e l'ordine della curva 
$f=f'_x=0$. 

E qui termino questa lunghissima lettera che vorrei potesse spingere a colmare 
nella teoria delle superficie quella lacuna che tutti avvertiamo. 

Cordiali saluti dal tuo aff.mo

\begin{flushright} 
GUIDO CASTELNUOVO
\end{flushright}

\vspace{1cm}

\begin{flushright} 
Rome, November 26, 1947
\end{flushright}

Dear Severi, 

following your wishes I am going to tell you some results about irregular surfaces; 
many of them are easily obtained and may be useful as exercises for your students. 
My ambitious and ultimate purpose was to build a theory of such surfaces independent 
of the notion of continuous system of curves, a theory embracing the theorem on the 
number $(p_g-p_a)$ of independent global differentials of the first kind, the theorem 
of Hodge, etc.. This program has just started; but the goal should be achieved, 
unless the theory of irregular surfaces hide amazing things I could not even imagine.   

\vspace{0.5cm}
1.

Let $\vert C \vert$ be a regular system of degree $n$ and genus $\pi$; in many cases
we need to assume that $\vert C \vert$ is \emph{sufficiently ample}, containing the 
canonical system $\vert K \vert$ or even one of its multiples; since we are looking 
for invariant characters, this is immaterial. I denote by $\chi$ your invariant $q'$, 
namely, the number of independent curves of $\vert 2C + K \vert$ passing through the 
jacobian group $G_\delta$ of a pencil $\vert C \vert$, hence through the base locus 
$G_n$ of the pencil. Let $G_k$ be the group of the $k$ cuspidal points\footnote{The usual 
English term for \emph{cuspidal points} is \emph{pinch points}.} of a surface, 
of degree $n$, with only ordinary singularities, and whose plane sections belong to 
the system $\vert C \vert$. 

Here is a meaning of $\chi$ which is immediate:

\vspace{0.3cm}
\noindent
1) \emph{The invariant $\chi$ is the superabundance of the system $\vert 4C + 2K \vert$ 
with respect to the cuspidal points $G_k$} (i.e. $G_k$ imposes $k -\chi$ conditions 
to such system).

On the other hand, $G_k$ gives independent conditions to the systems
$\vert mC + K \vert$, $\vert mC + 2K \vert$, \ldots, per $m \ge 5$.

For $m=4$ there is a result of Enriques, indirectly obtained by a moduli computation 
for a surface, but which should be directly proven; I recall it because it is coming 
into play shortly later: ``The superabundance of the system $\vert 4C + K \vert$ 
with respect to the group $G_k$ of cuspidal points is an invariant'', which I 
will denote $Q'$ and whose expression I am going to give below.

\vspace{0.3cm}
\noindent
2) \emph{The complete series $g_k$ determined by the group $G_k$ on a curve of 
$\vert 4C + K \vert$ passing through it has dimension $\chi$}.

\vspace{0.3cm}
\noindent
3) \emph{The complete series $g_k$ determined by the group $G_k$ on a curve of
$\vert 3C + K \vert$ passing through it (for instance: on $f=f'_x=0$) has dimension
$2n -\pi+2p_g+p_a-(I+4)+\theta$ where $0 \le \theta \le p_g-p_a$} (assume 
$\vert C \vert$ sufficiently ample). Is $\theta$ an invariant?

\vspace{0.5cm}
2.

The same argument which led you to establish the invariance of
$q'=\chi$ immediately shows:

\vspace{0.3cm}
\noindent 
4) \emph{It is invariant the number of linearly independent curves of 
$\vert 2C + 2K \vert$ passing through the jacobian group $G_\delta$ of a pencil 
$\vert C \vert$ and also through the base group of $G_n$}; I will denote 
this invariant by $Q$.

Then one sees (if the aforementioned result of Enriques is correct) that: 

\vspace{0.3cm}
\noindent 
5) \emph{It is also invariant the number of linearly independent curves of
$\vert 2C + 2K \vert$ passing through the group $G_\delta$ without having to pass 
through $G_n$}; this new invariant is equal to Enriques invariant $Q'$. 

It is therefore invariant the number of conditions that a curve of $\vert 2C + 2K \vert$
passing through the jacobian group $G_\delta$ of a pencil $\vert C \vert$ has to satisfy 
in order to contain the base group. One proves that this invariant satisfies the inequality
$Q'-Q \le p_g$. 

Regarding the expressions of $Q$ e $Q'$ I can state the following.

If the $\infty^{Q-1}$ curves of $\vert 2C + 2K \vert$ passing through $G_\delta + G_n$ cut
on a curve of $\vert 2C + K \vert$ passing through the same group a \emph{complete} series 
(residual of $G_\delta + G_n$ with respect to the canonical series) then: 
$$
Q-1 = p_a + p_g + p^{(1)} - (I-4) + \omega
$$
where $\omega$ ($\le p_g-p_a$) is a new invariant which has a very simple meaning: 
$I+4-\omega-1$ is the number of conditions that a group $G_{I+4}$ of your series 
of equivalence (in the strict sense) prescribes to the \emph{bicanonical} curves 
forced to contain it.  

If such a linear series is not complete, from the expression of $Q-1$ one has to subtract 
the deficiency $\le p_g-p_a$ of the series. 

In the same circle of ideas I also tell you the following result:

\vspace{0.3cm}
\noindent 
6) \emph{The superabundance of the system $\vert 3C + K \vert$ with respect to the jacobian 
group $G_\delta$ of a pencil $\vert C \vert$ is an invariant and its value is precisely 
$2p_g$} (if the complete system $\vert C \vert$ to which the pencil belong is sufficiently 
ample).

\vspace{0.5cm}  
3.

Other issues. As you know, the fundamental theorem to be proven is the following: 
a curve $\Gamma$ of $\vert 2C + K \vert$ passing through the jacobian group 
$G_\delta$ and the base group $G_n$ of a pencil $\vert C \vert$ cuts on the generic 
curve $C$ of the pencil (off $G_n$) a canonical group \emph{which does not belong to 
an adjoint curve  $C'$}. I tried to translate this condition into other equivalent 
formulations. Such is for instance the following one: the group $G_n$ on that curve 
has to impose independent conditions to the characteristic series of $\Gamma$ 
\emph{made complete}. To the characteristic series in the strict sense, $G_n$
imposes only $n-1$ conditions. 

Other formulation: Let us write Picard's identities in homogeneous coordinates: 
$$
Xf'_x+Yf'_y+Zf'_z+Tf'_t=0
$$
where $X=0, \ldots$ are surfaces of degree $n-3$. We have to add the condition 
(not explicitly stated) that the surfaces $yX-xY=0, \ldots, tZ-zT=0$ 
are adjoint to $f=0$ of degree $n$ \footnote{This is clearly an error, Castelnuovo means adjoint of degree $n-2$.}. It follows that $X=0, \ldots, T=0$ 
pass simply through the $t$ triple points of $f$ 
and moreover share a group of $(n-4)d-3t$ points on the double curve of degree 
$d$ of $f$; furthermore, they cut the planes $x=0, \ldots, t=0$, respectively, 
in adjoint curves of degree $n-3$. Hence it follows that such group of points 
of the double curve belongs to the series cut on this curve by the surfaces 
of degree $n-4$ passing simply through the $t$ triple points, provided this 
series has been made complete, while it has deficiency $p_g-p_a$ 
by the very definition of irregularity. 
Now, the fundamental theorem is equivalent to the following: 
\emph{Through such groups of $(n-4)d-3t$ points of the double curve 
it does not pass any surface of degree $n-4$ containing the $t$ triple points of $f$}\footnote{The right statement here would be: \emph{Through such groups of $(n-4)d-3t$ points of the double curve it does not pass any surface of degree $n-4$ containing the $t$ triple points of $f$ and not containing the double curve.}}. 
This statement translates into the following one, which is quite elegant from the 
analytic viewpoint: 
\emph{It is impossible to verify an identity of the form:
$$
\overline{X}f'_x+\overline{Y}f'_y+\overline{Z}f'_Z+\overline{T}f'_t 
\equiv Qf
$$
where $\overline{X}=0, \ldots, \overline{T}=0$ are adjoint surfaces of degree
$n-3$ and $Q=0$ is a surface of degree $n-4$ except in the trivial case 
$\overline{X}= \frac{1}{n}xQ, \ldots, \overline{T}= \frac{1}{n}tQ$}.
  
Going back to Picard's identity as written above, I suggest to you to propose to 
some student to investigate the homography between the system of non--adjoint degree $n-3$ surfaces
 $\lambda X + \mu Y + \nu Z + \rho T = 0$ and the system 
of planes $\lambda x + \mu y + \nu z + \rho t = 0$, each cutting the corresponding 
surface in an adjoint curve. In the case of irrational ruled surfaces of low degree 
one finds very elegant properties.  

\vspace{0.5cm}
4.

Finally, a few remarks you may find useful if you will present in a course your note 
about semiexact integrals. 

You prove that to every curve of $\vert 2C + K \vert$ passing through the jacobian group 
and the base group of a pencil $\vert C \vert$ (\emph{covariant curve of the pencil}, 
as I call it) one can associate a unique covariant curve of every other pencil 
$\vert D \vert$. Two associated curves cut on the same group of points the contact 
curve of two pencils. They moreover cut each other in a group $G_{I+4}$ 
of your series. One easily checks that this group is common to the whole family 
of associated covariant curves with respect to the infinitely many pencils on the 
surface. Every curve $C$ of $f=0$ is cut by the covariant curve of the family 
in a canonical group which I will call \emph{trace group}.

Taken a point $P$ of the surface, there exist infinitely many curves through $P$ 
such that $P$ belongs to the trace group. \emph{All these curves intersect in $P$}. 
It means that to every point $P$ of $f=0$ is associated a tangential direction, 
or a linear element (not defined only if $P$ belongs to the group $G_{I+4}$). 
By connecting all these elements the surface is covered by a pencil of (transcendental)
curves which turn out to be the integral curves of the equation $Bdx-Ady=0$, where $A=0$ 
and $B=0$ are two adjoint surfaces of degree $n-2$ cutting on $f$ the covariant curves 
of the pencils $x =$ const., $y =$ const. Of course one should show that $1/f'_z$ 
is an integral factor of the differential expression. 

Varying $P$ on $f$ the tangent in $P$ describes an algebraic congruence 
of class $2 \pi - 2$ and degree $k - \nu = 6 \pi - 6 + p^{(1)} - 1 - (I+4)$ 
where $k$ is the number of cuspidal points and $\nu$ is the degree of the curve 
$f=f'_x=0$. 

Here I stop this quite long letter I wish it could stimulate to fill in the theory 
of surfaces that gap we all perceive.

Best regards, yours friendly 

\begin{flushright} 
GUIDO CASTELNUOVO
\end{flushright}

\subsection{Guido Castelnuovo to Beniamino Segre} 

\begin{flushright} 
Roma, 15 genn. 50
\end{flushright}

Caro Professore, 

In relazione alla nostra conversazione di venerd\`i scorso e al programma di ricerche 
di cui Le parlavo, penso di sottoporle una questione, risolta la quale si sarebbe compiuto 
un passo notevole verso la meta cui Le accennavo. Si tratta di una questione di geometria 
algebrica, la quale, ove si possano togliere alcune restrizioni forse non necessarie, 
si muta in una questione relativa alle equazioni alle derivate parziali con condizioni 
al contorno. Con i mezzi svariati e potenti di cui Ella dispone potr\`a affrontarla 
e pervenire alla risposta desiderata. 

Sia $f(x,y,z,t)$ una superficie (in coord. omog.) d'ordine $n$, irriducibile, con 
singolarit\`a ordinarie; e siano $X=0$, $Y=0$, $Z=0$, $T=0$ quattro superficie 
aggiunte d'ordine $n-3$. Si tratta di dimostrare che un'identit\`a del tipo
\begin{equation} \label{one}  
X f'_x + Y f'_y + Z f'_z + T f'_t = Q f,
\end{equation}
con $Q$ polinomio di grado $n-4$, \underline{non} pu\`o sussistere salvo nel caso banale 
(identit\`a di Eulero) $X =\frac{1}{n}xQ, \ldots, T=\frac{1}{n}tQ$. Per farle vedere 
l'interesse della questione Le dir\`o che se si toglie la condizione che le sup. 
$X, \ldots, T$ siano aggiunte, e si sostituisce con la condizione meno stretta che 
siano aggiunte le sei superficie d'ordine $n-2$ $yX-xY=0, \ldots$, allora la identit\`a 
pu\`o sussistere con $Q$ identicamente nulla; anzi di identit\`a di quel tipo ve ne sono 
$p_g - p_a$ indipendenti per una superficie irregolare (Picard). 

Ritornando alla (\ref{one}), supposto che essa possa aver luogo, si vedrebbe che la 
superficie $Q=0$ incontra la curva doppia di $f=0$ nei punti tripli e nei punti ove 
$X'_x + Y'_y + Z'_z + T'_t = 0$, donde si concluderebbe che 
$Q \equiv X'_x + Y'_y + Z'_z + T'_t + \overline{Q}$ (salvo un fattore costante), 
essendo $\overline{Q}=0$ una superficie \underline{aggiunta} d'ordine $n-4$ che 
darebbe luogo a un integrale doppio senza periodi; e di qua l'assurdo (Hodge). 
Ma io richiedo evidentemente una dimostrazione pi\`u diretta e pi\`u elementare 
di quella qui abbozzata. 

Ci pensi quando ha tempo, perch\'e mi pare ne valga la pena. Cordiali saluti; aff.mo 

\begin{flushright} 
G. Castelnuovo
\end{flushright}

\vspace{1cm}

\begin{flushright} 
Rome, January 15, 1950
\end{flushright}

Dear Professor, 

Concerning our conversation of last Friday and the research program I exposed to you, 
I am going to propose to you a question, whose solution would provide a remarkable step 
towards the goal I mentioned. It is a question in algebraic geometry, which, up to 
removing some maybe unnecessary restrictions, translates into a question in 
partial differential equations with boundary conditions. By applying the many 
and poweful tools you have at your disposal you could address it and obtain the 
desired answer. 

Let $f(x,y,z,t)$ be a surfaces (in homogeneous coordinates) of degree $n$, irreducible, 
with ordinary singularities; let $X=0$, $Y=0$, $Z=0$, $T=0$ be four adjoint surfaces 
of degree $n-3$. The point is to show that an identity of the form
\begin{equation} \label{two}  
X f'_x + Y f'_y + Z f'_z + T f'_t = Q f,
\end{equation}
with $Q$ polynomial of degree $n-4$, is \underline{not} satisfied unless in the trivial case 
(Euler identity) $X =\frac{1}{n}xQ, \ldots, T=\frac{1}{n}tQ$. In order to show you the interest 
of the question I will tell you that if one drops the condition that the surfaces 
$X, \ldots, T$ are adjoint, and one replaces it by the less strict condition that 
the six degree $n-2$ surfaces  $yX-xY=0, \ldots$ are adjoint, then the identity 
may hold with $Q$ identically zero; indeed, there are $p_g - p_a$ independent 
such identities for an irregular surface (Picard). 

Coming back to (\ref{two}), assuming it may hold, one would see that the surface  
$Q=0$ meets the double curve of $f=0$ in the triple points and in the points where 
$X'_x + Y'_y + Z'_z + T'_t = 0$, whence one would conclude that 
$Q \equiv X'_x + Y'_y + Z'_z + T'_t + \overline{Q}$ (up to a constant factor), 
where $\overline{Q}=0$ would be an \underline{adjoint} surface of degree $n-4$ 
which would give rise to a double integral without periods, hence a contradiction 
(Hodge). But of course I am looking for a more direct and more elementary proof 
than the one sketched here. 

Please think about that when you have time, because I believe it is worth the trouble. 
Best regards; yours friendly

\begin{flushright} 
G. Castelnuovo
\end{flushright}

\section{Regular 1--forms on a surface}
\label{sec:reg}

If $X$ is a smooth, irreducible, projective surface over an algebraically closed field $\bbK$, the elements of $H^0(X,\Omega^1_X)$ are called \emph{regular 1--forms} on $X$. We will denote the dimension of $H^0(X,\Omega^1_X)$ by $q_{\rm an}(X)$ (or simply by $q_{\rm an}$ if there is no danger of confusion) and we will call it the \emph{analytic irregularity} of $X$ (see \cite {Cil}).

In this section we want to explain the background of Castelnuovo's letters, using  modern terminology. In particular we want to 
explain a classical method, very familiar to Castel\-nuovo and due to Picard and Severi, of constructing regular 1--forms on a surface. In  \S \ref {ssec:closed}, we explain Castelnuovo's viewpoint on the attempts of proving closedness of regular 1--forms on a surface.

\subsection{The general set up}\label{ssec:setup}

Let $X$ be a smooth, irreducible, projective surface over an algebraically closed field $\bbK$.  We may assume $X$ to be linearly normally embedded as a surface of degree $d$ in a projective space $\bbP^r$, with $r\geq 5$, in such a way that the following happens. If we consider a general projection $\pi$ of $S$ to $\bbP^3$, whose image is a surface $S$ of degree $d$, then $S$ has \emph{ordinary singularities} (see \cite [Thm. 2]{Ro}), i.e., it has:\\
\begin{inparaenum}
\item [$\bullet$] an irreducible \emph{nodal} double curve $\Gamma$, i.e., $S$ has normal crossings at the general point of $\Gamma$,\\
\item[$\bullet$]  a finite number of triple points for both $\Gamma$ and $S$, the triple points for $\Gamma$ are \emph{ordinary}, i.e., the tangent cone there to $\Gamma$ consists of the union of three non--coplanar lines, and the tangent cone there to $S$ consists of the union of three distinct planes,\\
\item [$\bullet$] finitely many \emph{pinch points} on $\Gamma$; we will denote by $G_c$ the \emph{pinch points scheme}, i.e., the reduced zero--dimensional scheme on $X$ where the differential of $\pi$ drops rank, so that $G_c$ is mapped by $\pi$ to the set of pinch points of $S$ on $\Gamma$. We will set $\gamma={\rm length}(G_c)$. 
\end{inparaenum}

The map $\pi: X\to S$ is the normalization map.

 We will introduce homogeneous coordinates $[x_1,x_3,x_3,x_4]$ in $\bbP^3$ and related affine coordinates $(x,y,z)$, with 
 $$
 x=\frac {x_1}{x_4}, \quad y=\frac {x_2}{x_4}, \quad z=\frac {x_3}{x_4}
 $$
 so that $x_4=0$ is the \emph{plane at infinity}. We assume that the coordinates (i.e., the corresponding fundamental points) are general with respect to $S$. The homogeneous equation of $S$ is of the form $F(x_1,x_2,x_3,x_4)=0$, with $F$ an irreducible homogeneous polynomial of degree $d$ and the affine equation of $S$ is $f(x,y,z)=0$, with $f(x,y,z)=F(x,y,z,1)$. We will denote by $f_x,f_y,f_z$ the partial derivatives of $f$ with respect to $x,y,z$ and by $F_i$ the partial derivative of $F$ with respect to $x_i$, for $1\leq i\leq 4$ (we will use similar notations for other polynomials). Note that
\begin{equation}\label{eq:der}
 f_x(x,y,z)=F_1(x,y,z,1)
\end{equation}
 and similarly for the other derivatives. Therefore, by Euler's identity, we have
\begin{equation}\label{eq:eul}
 d\cdot f(x,y,z)=xF_1(x,y,z,1)+yF_2(x,y,z,1)+zF_3(x,y,z,1)+F_4(x,y,z,1).
\end{equation}
 
By the generality assumption of the  coordinates  with respect to $S$ we have that:\\
 \begin{inparaenum}
 \item [$\bullet$] the plane at infinity  is not \emph{tangent} to $S$, i.e., it cuts out on $S$ a curve whose pull--back on $X$ via $\pi$ is smooth;\\
 \item [$\bullet$] each of the pencils $\calP_i$ of planes with homogeneous equations $hx_i=kx_4$, with $(h,k)\in \bbK\setminus \{(0,0)\}$, pulls back via $\pi$ to a \emph{Lefschetz pencil} $\calX_i$ on $X$, with $1\leq i\leq 3$;\\
  \item [$\bullet$] the pull--back $\Gamma_i$ on $X$ of the curve $\gamma_i$ cut out on $S$ off the double curve $\Gamma$ by the \emph{polar} surfaces $F_i=0$ is smooth for $1\leq i\leq 3$. 
 \end{inparaenum}
 
\begin{remark}\label{rem:var} By the genericity of the position of $S$ with respect to the coordinate system, one sees that the curves $\Gamma_i$, for $1\leq i\leq 3$, contain the pinch points scheme $G_c$ and  intersect pairwise transversely there. 

The singular points of the finitely many curves in the pencil $\calX_i$ are nodes and form a reduced 0--dimensional scheme $\calJ_i$ on $X$, which is called the \emph{jacobian scheme} of $\calX_i$, for $1\leq i\leq 3$. We will assume that, for all $i\in \{1,2,3\}$, the image $J_i$  of this scheme on $S$, called the  \emph{jacobian scheme} of $\calP_i$, has no intersection with the double curve $\Gamma$. 

It is also easy to check that the curve $\Gamma_i$ cuts out on $\Gamma_j$ the divisor $G_c+\calJ_k$, where $\{i,j,k\}=\{1,2,3\}$. So, in particular, taking into account that
$\Gamma_i\in |3C+K_X|$, for $1\leq i\leq 3$ (we denote by $C$ a hyperplane section of $X$), one has
\begin{equation}\label{eq:uno}
\calO_{\Gamma_3}(G_c+\calJ_1)=\calO_{\Gamma_3}(G_c+\calJ_2)=\calO_{\Gamma_3}(3C+K_X),
\end{equation}
hence
\begin{equation}\label{eq:due}
 \calO_{\Gamma_3}(\calJ_1)=\calO_{\Gamma_3}(\calJ_2).
\end{equation}
Similar relations hold on $\Gamma_2$ and $\Gamma_3$. Note that \eqref {eq:uno} implies that $|3C+K_X|$ has no fixed component and $(3C+K_X)^2>0$, hence $3C+K_X$ is big and nef.
\end{remark}

Let $e:=e(X)$ be the \emph{Euler--Poincar\'e characteristic} of $X$ (i.e., the second \emph{Chern class} of the tangent bundle of $X$) and $g$ the arithmetic genus of the hyperplane sections of $X$. By the \emph{Zeuthen--Segre formula} (see \cite [p. 301]{Fu}), the length $\delta$ of $J_i$ is
 $$
 \delta=e+4(g-1)+d.
 $$

 \subsection{The expression of 1--forms on a surface}
 It is a result by Picard (see \cite [p. 116]{PS}, Picard works over $\bbC$ but it is easy to check that his argument works on any algebraically closed field $\bbK$) that if $\omega$ is  a regular 1--form on $X$, then it is the pull--back on $X$ of a rational 1--form of the type
\begin{equation}\label{eq:1form}
\frac {Ady-Bdx}{f_z}
\end{equation}
 where $A=0,B=0$ are affine equations of two \emph{adjoint surfaces}  of degree $d-2$ to $S$. Recall that a surface is said to be adjoint to $S$ if it contains the double curve $\Gamma$ of $S$.
 
 In the 1--form \eqref {eq:1form} we can make a change of variables passing from $x,y$ to $x,z$. From the relation
 $$
 f_xdx+f_ydy+f_zdz=0
 $$
 we deduce
 $$
 dy=-\frac {f_xdx+f_zdz}{f_y}.
 $$
Substituting into \eqref {eq:1form} we find
 $$
\frac{-\frac {Af_x+Bf_y}{f_z}dx -Adz}{f_y}
 $$
 and this has to be of the same form as \eqref {eq:1form} with respect to the variables $x,z$. This implies that there must be a polynomial $C$ of degree $d-2$ such that $C=0$ is the affine equation of an adjoint surface to $S$, such that
 $$
 -\frac {Af_x+Bf_y}{f_z}=C, \quad \text{modulo}\quad f=0.
 $$
 This yields the  \emph{Picard's relation}
\begin{equation}\label{eq:pic}
 Af_x+Bf_y+Cf_z=Nf
\end{equation}
 where $N$ is a suitable polynomial of degree $d-3$. The Picard relation has some remarkable consequences, pointed out by Severi (see \cite [\S 9]{Sev28}). Before stating Severi's result, we recall the following:
 
 \begin{lemma}[Castelnuovo's Lemma] \label{lem:cast} Let $g(x_1,x_2,x_3)=0$ be the equation of an irreducible plane curve of degree $n$ with no singular points except nodes. Then there is no non--trivial syzygy of degree $l\leq d-2$ of the triple $(g_1,g_2,g_3)$ of derivatives of $g$. 
 \end{lemma}
 
 For the proof see \cite [\S 7]{Sev28} or \cite [p. 34]{Ka}. Next we can prove Severi's result:
 
 \begin{proposition}\label{prop:sev1} If $A,B,C$ are non--zero polynomials verifying \eqref {eq:pic}, then the (projective closure of the) surface with equation $A=0$ [resp. $B=0$, $C=0$] contains the base line of the pencil of planes $\calP_1$ [resp. of $\calP_2$, of  $\calP_3$] and also the jacobian scheme $J_1$ [resp. $J_2$, $J_3$] of this pencil. Moreover the (projective closures of the) surfaces $A=0, B=0, C=0$ cut out on the plane at infinity the same curve off the aforementioned lines. 
 \end{proposition}
 
 \begin{proof}  First we prove that the surface with equation $A=0$ contains the scheme $J_1$. 
 Let $P$ be a point of $J_1$. Then $f, f_y,f_z$ vanish at $P$. Hence by \eqref {eq:pic}, also $Af_x$ vanishes at $P$. However $f_x$ does not vanish at $P$ because $P$ does not belong to the double curve $\Gamma$ of $S$. Hence $A$ vanishes at $P$. Similarly for the surface with equation $B=0$ [resp. $C=0$] containing the scheme $J_2$ [resp. $J_3$]. 
 
 Next, homogenize \eqref {eq:pic}. By \eqref {eq:der} (and the similar for the other derivatives) we get a relation of the form
 $$
 \bar A F_1+\bar BF_2+\bar CF_3=\bar NF
 $$
 where we denote by the bars the homogenization of the corresponding polynomials. By taking into account the Euler identity, this relation takes the form
 $$
 (d\bar A-x_1\bar N)F_1+(d\bar B-x_2\bar N)F_2+(d\bar C-x_3\bar N)F_3=x_4\bar NF_4.
 $$
 Setting $x_4=0$ and taking into account Castelnuovo's Lemma \ref {lem:cast}, we have identically 
 $$
 d\bar A-x_1\bar N\equiv 0, \quad d\bar B-x_2\bar N\equiv 0, \quad d\bar C-x_3\bar N\equiv 0
 $$
 under the condition $x_4=0$. This implies that 
 $$
 dA_0-x_1 N_0\equiv 0, \quad dB_0-x_2 N_0\equiv 0, \quad d C_0-x_3 N_0\equiv 0
 $$
 where $A_0,B_0,C_0$ are the homogeneous components of $A,B,C$ in degree $d-2$ and $N_0$ is the homogeneous component of $N$ of degree $d-3$. The assertion follows right away. \end{proof} 
 
 \begin{remark}\label{rem:lop} Note that  the surfaces with equations $A=0, B=0, C=0$ in Proposition \ref{prop:sev1} are not necessarily adjoint. Keeping the notation of the proof of Proposition \ref{prop:sev1}, set $N_0=\vartheta$. Then we have identities of the form
\begin{equation}\label{eq:iden}
 A=x\vartheta +A_1, \quad B=y\vartheta+B_1, \quad C=z\vartheta+C_1, \quad N=d\vartheta +N_1
\end{equation}
 where $A_1,B_1,C_1$ are (non--homogeneous) polynomials of degree at most $d-3$ and $N_1$ has degree at most $d-4$.
  \end{remark}
 
 Severi next proved the following proposition (see \cite[\S 9]{Sev28}):
 
 \begin{proposition}\label{prop:sev2} Let $A=0$ be the affine equation of an adjoint surface of degree $d-2$ to $S$ containing the scheme $J_1$. Then there are uniquely determined adjoint surfaces of degree $d-2$ to $S$ with affine equations $B=0$ and $C=0$, containing the schemes $J_2$ and $J_3$ respectively, such that \eqref {eq:pic} holds. Each of the polynomials $A,B,C$ uniquely determines the other two.
 \end{proposition}
 
 \begin{proof} Consider $A$ as in the statement.  
The complete linear system $|2C+K_X|$ is the pull--back to $X$ of the curves cut out on $S$, off the double curve $\Gamma$, by the adjoint surfaces of degree $d-2$. Looking at the exact sequence
$$
0\longrightarrow \calO_X(-C)\longrightarrow \calO_X(2C+K_X)\longrightarrow \calO_{\Gamma_3}(2C+K_X)\longrightarrow 0
$$
we see that $|2C+K_X|$ cuts out on $\Gamma_3$ a complete linear series $\xi$, because $h^1(X, \calO_X(-C))=0$ (by the Kodaira vanishing theorem, see \cite[p. 154]{GH}). Moreover, since $h^0(X, \calO_X(-C))=0$, the restriction map
$$
H^0(X,\calO_X(2C+K_X))\longrightarrow H^0(\Gamma_3, \calO_{\Gamma_3}(2C+K_X))
$$
is injective. 

Let us abuse notation and denote by $A\in |2C+K_X|$ the pull back on $X$ of  the curve cut out on $S$ by the (projective closure of the) surface $A=0$ off $\Gamma$. Then $A$ cuts out on $\Gamma_3$ a divisor of the form $\calJ_1+Z\in \xi$. Since $\calJ_2+Z\in \xi$ by \eqref {eq:due},  there is a unique curve $B\in |2C+K_X|$ that cuts out  $\calJ_2+Z$ on $\Gamma_3$. By abusing notation, we denote by $B$ a non--zero polynomial, uniquely defined up to a constant, such that $B=0$ is the adjoint surface cutting out on $S$ off $\Gamma$ the curve whose  pull--back on $X$ is $B$. The surfaces $Af_x$ and $Bf_y$ cut out on the curve $\gamma_3$ the same divisor, hence there is a non--zero constant $b$  such that $Af_x-bBf_y=0$ on $\gamma_3$. By substituting $B$ with $-bB$ we may assume that $Af_x+Bf_y=0$ on $\gamma_3$. 

Consider now the complete intersection scheme $Y$, whose ideal is generated by $f$ and $f_z$, which consists of two components given by $\gamma_3$ and by $\Gamma$ with a double structure. Since $Af_x+Bf_y$ vanishes on $\gamma_3$ and vanishes with multiplicity 2 on $\Gamma$, then $Af_x+Bf_y$ vanishes on $Y$ and therefore $Af_x+Bf_y$ is a  combination of $f$ and $f_z$, i.e., there are polynomials $C$ and $N$, of degrees $d-2$ and $d-3$ respectively, such that \eqref {eq:pic} holds. Note that $C$ cannot be identically zero. Otherwise we would have an identity of the sort
$$
Af_x+Bf_y=Nf.
$$
This is impossible, because then $Bf_y$ would vanish along the curve $\gamma_1$, but neither $f_y$ nor $B$ can vanish along this curve. Since $Af_x$, $Bf_y$ and $f$ vanish doubly along $	\Gamma$, then $C$ vanishes along $\Gamma$ so that $C=0$ is adjoint to $S$. Moreover $C$ is uniquely determined. In fact, from another identity of the form
$$
Af_x+Bf_y+C'f_z=N'f,
$$
subtracting memberwise from \eqref {eq:pic}, we deduce
$$
(C-C')f_z=(N-N')f
$$
and $f$ would divide the left hand side, what is impossible because both factors there have degree smaller than $f$. The  assertion follows.\end{proof}
 
By taking into account Proposition \ref {prop:sev1}, one has the:
 
\begin{corollary}\label{cor:sev1} Every adjoint surface to $S$ of degree $d-2$ containing the scheme $J_1$ contains also the base line of the pencil $\calP_1$.
\end{corollary}
 
We can state this corollary in an intrinsic form:
  
\begin{corollary}\label{cor:sev2} Let $X$ be a smooth, irreducible, projective surface, $C$ a very ample effective divisor on $X$ and $\calP$ a Lefschetz pencil in $|C|$. Then any curve in $|2C+K_X|$ containing the \emph{jacobian scheme} of the pencil $\calP$ (i.e., the scheme of double points of the singular curves in $\calP$) also contains the \emph{base locus scheme} of $\calP$.
 \end{corollary}
 
Now, given an adjoint surface of degree $d-2$ to $S$ containing the scheme $J_1$, with affine equation $A=0$, consider the other two adjoint surfaces $B=0$ and $C=0$ existing by Proposition \ref {prop:sev2}. We can consider the three regular 1--forms pull backs on $X$ of the forms
$$
\frac {Ady-Bdx}{f_z}, \quad \frac {Bdz-Cdy}{f_x}, \quad \frac {Cdx-Adz}{f_y}.
$$
By the very proof of Proposition \ref {prop:sev1} we see that these forms are equal. In conclusion, if we consider the vector space ${\rm Adj}_{d-2}(S)$ of (non--homogeneous) polynomials of degree (at most) $d-2$ defining adjoint surfaces to $S$ passing through $J_1$, this determines an isomorphism 
\begin{equation}\label{eq:isomat}
\varphi: {\rm Adj}_{d-2}(S)\to H^0(X, \Omega^1_X).
\end{equation}
The map $\varphi$ sends a polynomial $A$ to the 1--form pull--back of the form \eqref {eq:1form} to $X$, where $B=0$ is the adjoint surface of degree $d-2$ described in Proposition \ref {prop:sev2}. The same by exchanging $J_1$ with $J_2$ or $J_3$. 

\subsection{Closedness of 1--forms}\label{ssec:closed}  The following result is well known:

\begin{proposition}\label{prop:closed} If $X$ is a complex, smooth, compact surface, any regular 1--form on $X$ is closed.
\end{proposition}

\begin{proof} This proof is extracted from \cite [p. 137--138]{BHPV}.

Let $\omega\in H^0(X, \Omega^1_X)$ be a non--zero regular form. By Stokes' Theorem one has
\begin{equation}\label{eq:stok}
\int_X d\omega \wedge d\bar \omega=\int_X d(\omega\wedge d\bar\omega)=0.
\end{equation}
Write down locally $d\omega=fdz_1\wedge dz_2$. Then
$$
d\omega \wedge d\bar \omega=-|f|^2dz_1\wedge d\bar z_1\wedge dz_2\wedge d\bar z_2=
4 |f|^2 dx_1\wedge dy_1\wedge dx_2\wedge dy_2
$$
where $z_j=x_j+iy_j$, for $1\leq j\leq 2$, so that by \eqref {eq:stok} one gets $f=0$, i.e., $d\omega=0$.
\end{proof}

The proof of this proposition is analytic and does not hold in positive characteristic. In fact in positive characteristic there are counterexamples to Proposition \ref {prop:closed} (see \cite [Corollary] {Mu}). There is then the problem, which was classically well know (see \cite [p. 185]{Za}) and considered also in the two letters by Castelnuovo, of finding a purely algebraic proof of Proposition \ref {prop:closed}. It is useful for us to  review the classical viewpoint on this subject.

Let us keep the notation introduced above. Let $\omega$ be a regular 1--form on the surface $X$,  which is the pull--back on $X$ of the rational 1--form \eqref {eq:1form}. Then we have $d\omega=\phi dx \wedge dy$, with
$$
\phi=\frac \partial {\partial x}\Big (\frac A {f_z} \Big )+\frac \partial {\partial y}\Big (\frac B {f_z} \Big )
$$
where it is intended that the differentiations take place on the surface $X$, so that $z$ is function of $x,y$ implicitly defined by $f(x,y,z)=0$. So, for instance
$$
\frac {\partial z}{\partial x}=-\frac {f_x}{f_z}
$$
and 
$$
\begin{array}{cc}
\frac \partial {\partial x} \Big (\frac A {f_z} \Big )&=\frac {\Big (A_x+A_z\frac {\partial z}{\partial x}\Big)f_z-A\Big(f_{zx}+f_{zz}\frac {\partial z}{\partial x}\Big)}{f_z^2}=\\
&=\frac {\Big (A_x-A_z\frac {f_x}{f_z}\Big)f_z-A\Big(f_{zx}-f_{zz}\frac {f_x}{f_z}\Big)}{f_z^2}=\\
&=\frac {f_z^2A_x-f_z(Af_{zx}+A_zf_x)+f_{zz}Af_x}{f_z^3}	\\
\end{array}
$$
and similarly 
$$
\frac \partial {\partial y}\Big (\frac B {f_z} \Big )=\frac {f_z^2B_y-f_z(Bf_{zy}+B_zf_y)+f_{zz}Bf_y}{f_z^3}
$$ 
so that
\begin{equation}\label{eq:wurt}
\phi=\frac {f_z^2(A_x+B_y)-f_z(Af_{zx}+A_zf_x+Bf_{zy}+B_zf_y)+f_{zz}(Af_x+Bf_y)}{f_z^3}.
\end{equation}
Taking into account \eqref {eq:pic} and the identity
$$
\frac {\partial (Af_x+Bf_y)}{\partial z}=A_zf_x+Af_{xz}+B_zf_y+Bf_{yz},
$$
 \eqref {eq:wurt} becomes
$$
\phi=\frac 1{f_z^3}\Big [ f_z^2(A_x+B_y+C_z-N)+f(Nf_{zz}-f_zN_z) \Big ]
$$
so that 
$$
\phi=\frac {A_x+B_y+C_z-N}{f_z}, \quad \text{modulo}\quad f
$$
and this is regular on $X$. Hence if we set
$$
Q=A_x+B_y+C_z-N
$$
the polynomial $Q$ has to vanish on the double curve $\Gamma$ of $S$, because it has to vanish where $f_z$ vanishes. 

A priori $Q$ is a polynomial of degree $d-3$ but one has actually:

\begin{lemma}\label{lem:add} In the above setting $Q$ has degree $d-4$. 
\end{lemma}

\begin{proof} By taking into account the identities \eqref {eq:iden} in Remark \ref {rem:lop}, we have
$$
A_x=\theta+x\theta_x+\frac {\partial A_1}{\partial x}, \quad 
B_y=\theta+y\theta_y+\frac {\partial B_1}{\partial y},\quad 
C_z=\theta+z\theta_z+\frac {\partial C_1}{\partial z}
$$
where $\theta$ is a homogeneous polynomial of degree $d-3$. Hence, by Euler's identity, we get
\begin{equation}\label{eq:tre}
\begin{array}{cc}
A_x+B_y+C_z-N&=d\theta+\frac {\partial A_1}{\partial x}+\frac {\partial B_1}{\partial y}+\frac {\partial C_1}{\partial z}-(d\theta+N_1)=\\
&=\frac {\partial A_1}{\partial x}+\frac {\partial B_1}{\partial y}+\frac {\partial C_1}{\partial z}-N_1\\
\end{array}
\end{equation}
which proves the assertion. \end{proof}

In conclusion, we have
$$
\frac \partial {\partial x}\Big (\frac A {f_z} \Big )+\frac \partial {\partial y}\Big (\frac B {f_z} \Big )=\frac {Q}{f_z}
$$
and with similar computations one finds
$$
\frac \partial {\partial y}\Big (\frac B {f_x} \Big )+\frac \partial {\partial z}\Big (\frac C {f_x} \Big )=\frac {Q}{f_x},\quad  \frac \partial {\partial z}\Big (\frac C {f_y} \Big )+\frac \partial {\partial x}\Big (\frac A {f_y} \Big )=\frac {Q}{f_y}.
$$

In any event, the form $\omega$ as above is closed if and only if $Q=0$ modulo $f$. But, since $Q$ has degree smaller than $d$, this is the case if and only if $Q$ is identically zero. So, taking into acccount \eqref {eq:tre},
the problem of giving an algebraic proof of Proposition \ref {prop:closed} translates in the following:

\begin{problem}\label{prob:non} Find an algebraic proof that \eqref {eq:pic} implies either one of the two equivalent relations
\begin{equation}\label{eq:integ}
N=A_x+B_y+C_z, \quad  N_1 =\frac {\partial A_1}{\partial x}+\frac {\partial B_1}{\partial y}+\frac {\partial C_1}{\partial z}
\end{equation}
each of which is called the \emph{integrability condition}. 
\end{problem}

We want to stress that any solution of Problem \ref {prob:non} must use the fact that the base field $\bbK$ has characteristic zero. 

\subsection{Homogeneous form of Picard's relation}

It is useful  to describe the homogeneous form of Picard's relation \eqref {eq:pic}. This is contained in \cite [p. 119]{PS} and we expose this here for the reader's convenience. 

By \eqref {eq:eul}, we can rewrite \eqref {eq:pic} as
$$
\begin{array}{cc}
&d\cdot AF_1(x,y,z,1)+d\cdot BF_2(x,y,z,1)+d\cdot CF_3(x,y,z,1)=\\
&=N(xF_1(x,y,z,1)+yF_2(x,y,z,1)+zF_3(x,y,z,1)+F_4(x,y,z,1))\\
\end{array}
$$
Set
$$
X_1=\overline {dA-xN}, \quad X_2=\overline{dB-yN}, \quad X_3=\overline{dC-zN}, \quad X_4=-\bar N
$$
where, as usual, the bars stay for homogenization. By \eqref {eq:iden}, we have
$$
X_1=\overline {dA_1-xN_1}, \quad X_2=\overline{dB_1-yN_1}, \quad X_3=\overline{dC_1-zN_1}, \quad X_4=-\bar N
$$
and the polynomials $X_i$, with $1\leq i\leq 4$, are  of degree $d-3$. Then we have the relation
\begin{equation}\label{eq:pic2}
X_1F_1+X_2F_2+X_3F_3+X_4F_4=0
\end{equation}
which is the \emph{homogeneous Picard's relation}. If we consider the matrix
\begin{equation}\label{eq:matr}
M=\left (
\begin{array}{cccc}
X_1&X_2&X_3&X_4\\
x_1&x_2&x_3&x_4\\
\end{array}
\right),
\end{equation}
all minors of order 2 of $M$, after dehomogenization, are linear combinations of $A,B,C$ and so are in ${\rm Adj}_{d-2}(S)$. 

Suppose  the homogeneous Picard's relation \eqref {eq:pic2} holds. Taking into account the expressions of the polynomials $X_i$, for $1\leq i\leq 4$, and the relations 
\eqref {eq:iden}, the integrability relation in the form of the right hand side of  \eqref {eq:integ}, becomes
\begin{equation}\label{eq:integ2}
\frac {\partial X_1}{\partial x_1}+\frac {\partial X_2}{\partial x_2}+\frac {\partial X_3}{\partial x_3}+\frac {\partial X_4}{\partial x_4}=0
\end{equation}
which is the \emph{homogeneous integrability condition}.  Problem \ref {prob:non} can now be expressed in homogeneous form as:

\begin{problem}\label{prob:non2} Find an algebraic proof that \eqref {eq:pic2} (with all minors of order 2 of the matrix $M$ in \eqref {eq:matr}, after dehomogenization, in ${\rm Adj}_{d-2}(S)$) implies the homogeneous integrability condition \eqref {eq:integ2}.
\end{problem}

\section{Comments on Castelnuovo's letters}\label{sec:comments}

This section is devoted to explaining most of the issues raised by Castelnuovo in his two letters. Both letters focus on the understanding of the algebro--geometric meaning of the analytic irregularity $q_{\rm an}$ and on solving Problems \ref {prob:non} or \ref {prob:non2}. 

In \S\S 1 and 2 of the first letter, Castelnuovo suggests, with no proofs, various geometric interpretations of $q_{\rm an}$. Analogous remarks have been partially included by Castelnuovo  in the paper \cite {Cas} published two years after this letter. Castelnuovo does not say it, but maybe he had in mind in the letter that the various geometric interpretations of $q_{\rm an}$ could have been useful to algebro--geometrically prove the equality between $q_{\rm an}$ and $q_a:=h^1(X, \calO_X)$, that we will call the the \emph{arithmetic irregularity}, an equality that  Castelnuovo proved with analytic methods in the paper \cite {Ca1} of 40 years before (a different proof was given by Severi in \cite {Sev1};  see also \cite{Cil}). Note that this equality does not hold in positive characteristic, as proved by Igusa in \cite {Ig} (see also \cite {Mumf}). 

Let us keep the notation introduced so far.
The first result Castelnuovo states in his letter to Severi is the following:

\begin{proposition}\label{pro:cast1} Let $|C|$ be a very ample linear system on $X$. Then
$$
h^1(X, \calO_X(4C+2K_X)\otimes \calI_{G_c|X})=q_{\rm an}.
$$
\end{proposition}

\begin{proof} In Remark \ref {rem:var} we saw that $3C+K_X$ is big and nef. This  implies that also $4C+K_X$ is big and nef.

Look  at the exact sequence
$$
0\longrightarrow \calO_X(C+K_X)\longrightarrow \calO_X(4C+2K_X)\longrightarrow \calO_{\Gamma_3}(4C+2K_X)\longrightarrow   0.
$$
We have $h^i(X, \calO_X(C+K_X))=0$ for $1\leq i\leq 2$ (by the Kodaira vanishing theorem), and $h^1(X,\calO_X(4C+2K_X))=0$, because $4C+K_X$ is big and nef (by Mumford's theorem, see \cite[\S II]{Mu2}). This implies that $|4C+2K_X|$ cuts out on $\Gamma_3$ a complete, non--special linear series $g^r_n$, where
$$
r=n-p_a(\Gamma_3)
$$
(recall the definition of the curves $\Gamma_i$, $i=1,2,3$, from the beginning of \S \ref {ssec:setup}). 

Set now 
$$
h^1(X, \calO_X(4C+2K_X)\otimes \calI_{G_c|X})=h.
$$
The linear system $|\calO_X(4C+2K_X)\otimes \calI_{G_c|X}|$ cuts out on $\Gamma_3$, off $G_c$, a complete linear series $\xi=g_{n-\gamma}^{r-\gamma+h}$ (recall that $\gamma={\rm length}(G_c)$), so that $h$ is the index of speciality of $\xi$. 

Let $G$ be a general divisor of $\xi$, so that
$$
\calO_{\Gamma_3}(G+G_c)=\calO_{\Gamma_3}(4C+2K_X).
$$
Let $G'$ be a divisor on $\Gamma_3$ such that 
$$
\calO_{\Gamma_3}(G'+\calJ_1)=\calO_{\Gamma_3}(2C+K_X).
$$
Adding up these two relations and subtracting \eqref {eq:uno}, we get
$$
\calO_{\Gamma_3}(G+G')=\calO_{\Gamma_3}(3C+2K_X)=\omega_{\Gamma_3}.
$$
So we get 
$$
h=h^0(\Gamma_3, \calO_{\Gamma_3}(G')).
$$
On the other hand, by looking at the exact sequence
$$
0\longrightarrow \calO_X(-C)\longrightarrow \calO_X(2C+K_X)\longrightarrow \calO_{\Gamma_3}(2C+K_X))\longrightarrow   0,
$$
since $h^1(X, \calO_X(-C))=0$ (by the Kodaira vanishing theorem), we see that $|2C+K_X|$ cuts out on $\Gamma_3$ a complete linear series, hence 
$$
h=h^0(\Gamma_3, \calO_{\Gamma_3}(G'))=h^0(X, \calO_X(2C+K_X)\otimes \calI_{\calJ_1|X})
$$
and the assertion follows by the isomorphism $\varphi$ in \eqref {eq:isomat}.
\end{proof}

After this Castelnuovo claims that 
$$
h^1(X, \calO_X(nC+mK_X)\otimes \calI_{G_c|X})=0
$$
for $n\geq 5$ and $m\geq 1$. We have not been able to prove (or disprove) this assertion.

Another geometric interpretation of $q_{\rm an}$ that Castelnuovo suggests in the letter to Severi is the following: let $D$ be a curve in $|4C+K_X|$ that contains $G_c$ and it is smooth there, then 
\begin{equation}	\label{eq:err}
h^0(D, \calO_D(G_c))=q_{\rm an}+1.
\end{equation}
Also for this statement we could not come up with a proof (or a counterexample). 

\begin{remark}\label{rem:err} It looks rather difficult that \eqref {eq:err} could hold. In fact, consider again the curve $\Gamma_3$. Then $D$ cuts out on $\Gamma_3$ a divisor $G_c+G$, where,  by \eqref {eq:uno}, one has
\begin{equation}\label{eq:pul0}
\calO_{\Gamma_3}(G)=\calO_{\Gamma_3}(\calJ_1+H)
\end{equation}
where $H$ is a divisor cut out on $\Gamma_3$ by a hyperplane. By looking at the exact sequence
$$
0\longrightarrow \calO_X(-C)\longrightarrow \calO_X(3C+K_X)\longrightarrow  \calO_D(3C+K_X)\longrightarrow 0
$$
and since $h^1(X, \calO_X(-C))=0$, we see that $|3C+K_X|$ cuts out on $D$ a complete linear series. Hence the linear series $| \calO_D(G_c)|$ is cut out on $D$, off $G$, by the linear system $|\calO_X(3C+K_X)\otimes \calI_{G|X}|$, and therefore
\begin{equation}\label{eq:pul}
h^0(D, \calO_D(G_c))=h^0(X,\calO_X(3C+K_X)\otimes \calI_{G|X}).
\end{equation}
From the exact sequence 
$$
0\longrightarrow \calO_X\longrightarrow \calO_X(3C+K_X)\otimes \calI_{G|X}\longrightarrow \calO_{\Gamma_3}(3C+K_X)\otimes \calI_{G|X}\longrightarrow 0
$$
we have
\begin{equation}\label{eq:pul1}
h^0(X,\calO_X(3C+K_X)\otimes \calI_{G|X})\leq h^0(\Gamma_3, \calO_{\Gamma_3}(3C+K_X)\otimes \calI_{G|X})+1.
\end{equation}
By \eqref {eq:pul0}, we have
$$
\calO_{\Gamma_3}(3C+K_X)\otimes \calI_{G|X}=\calO_{\Gamma_3}(2C+K_X)\otimes \calI_{\calJ_1|X}.
$$
Since, as we saw in the proof of Proposition \ref {pro:cast1}, $|2C+K_X|$ cuts out on $\Gamma_3$ a complete linear series, we have
$$
h^0(\Gamma_3,\calO_{\Gamma_3}(2C+K_X)\otimes \calI_{\calJ_1|X})=
h^0(X,\calO_{X}(2C+K_X)\otimes \calI_{\calJ_1|X})=q_{\rm an}.
$$
Putting together this, \eqref {eq:pul} and \eqref {eq:pul1}, one gets
$$
h^0(D, \calO_D(G_c))\leq q_{\rm an}+1.
$$
Now the equality holds if and only if the restriction map
$$
H^0(X, \calO_X(3C+K_X)\otimes \calI_{G|X})\longrightarrow H^0(\Gamma_3,\calO_{\Gamma_3}(3C+K_X)\otimes \calI_{G|X})
$$
is surjective. This looks difficult because the map
$$
H^0(X, \calO_X(3C+K_X))\longrightarrow H^0(\Gamma_3,\calO_{\Gamma_3}(3C+K_X))
$$
is not surjective (it has corank $q_a$).
\end{remark}

At the end of the first section of his letter to Severi, Castelnuovo claims that: if $D$ is a curve in $|3C+K_X|$ containing $G_c$ and smooth there, then
$$
h^0(D,\calO_D(G_c))=2d-g+2p_g+\chi-e-1+\theta
$$ 
with $0\leq \theta\leq q_a$, and, as usual, $\chi=\chi(\calO_X)$ and $p_g=h^0(X,\calO_X(K_X))$. It is easy to check that this is equivalent to 
$$
2p_g+\chi-1\leq h^1(D,\calO_D(G_c))\leq 3p_g.
$$
However we have not been able to prove this. 

In the second section of the letter to Severi, Castelnuovo states the:

\begin{proposition}\label{prop:cas3} Let $\calP$ be a Lefschetz pencil in $|C|$, with jacobian scheme $\calJ$. Then 
$$
h^1(X,\calO_X(4C+K_X)\otimes \calI_{G_c|X})=h^0(X,\calO_X(2C+2K_X)\otimes \calI_{\calJ|X}).
$$
\end{proposition}

The proof of this, not so different from the one of Proposition \ref {pro:cast1}, is contained in \cite [\S 3]{Cas} and we do not reproduce it here. Let $B$ be the  base locus scheme of the Lefschetz pencil $\calP$. Castelnuovo compares $h^0(X,\calO_X(2C+2K_X)\otimes \calI_{\calJ|X})$ with $h^0(X,\calO_X(2C+2K_X)\otimes \calI_{\calJ+B|X})$. This does not look particularly interesting and we do not dwell on it here. Castelnuovo also claims that if $\calP$ is a Lefschetz pencil in $|C|$ with jacobian scheme $\calJ$, then 
$$
h^1(X,\calO_X(3C+K_X)\otimes \calI_{\calJ|X})=2p_g
$$
but we have not been able to prove it.

Let us jump for a moment to section 4 of the letter to Severi. In this part, as well as in the paper \cite{Cas}, Castelnuovo takes for granted the existence of the so called \emph{Severi equivalence series}. Severi claimed in \cite{Sev32} that, unless the surface has an irrational pencil, there exists the rational equivalence series, of dimension $q_{\rm an}-1$, of the zero dimensional schemes of length $e$ that are zeros of non--zero 1--forms in $H^0(X,\Omega^1_X)$. In addition Severi claimed that 
$\Omega^1_X$ is generated by global sections. These claims are  false  in general, as shown by F. Catanese in \cite [\S 6] {Cat}. Hence the contents of section 4 of the letter, and of 
the paper \cite{Cas} have biases because of this.

Let us now go back to section 3 of the letter to Severi. The focus of this section is on Problems 
 \ref {prob:non} or \ref {prob:non2}. Castelnuovo proposes a few equivalent formulations of these problems, the most interesting of which, in our opinion, is the following, which is also the topic of the letter to B. Segre. 
 
 \begin{problem}
\label{prob:gral} Suppose there is a (homogeneous) relation of the form
\begin{equation}\label{eq:integ3}
Y_1F_1+Y_2F_2+Y_3F_3+Y_4F_4=QF
\end{equation}
where $Y_i=0$ are adjoint surfaces of degree $d-3$ to $S$ and $Q=0$ is a surface of degree $d-4$. Prove (algebraically) that $Q$ is an adjoint surface and that
$$
Y_i=\frac 1d x_iQ.
$$
\end{problem} 

The solution of this problem implies the solution of  Problem \ref {prob:non2}. Indeed,  we can rewrite \eqref {eq:integ3} as
$$
\sum_{i=1}^4\Big (Y_i-\frac 1d Qx_i\Big) F_i=0
$$
and this is a homogeneous Picard's relation of the type \eqref {eq:pic2}, with
$$
X_i=Y_i-\frac 1d Qx_i, \quad 1\leq i\leq 4.
$$
Problem \ref {prob:non2} asks to prove that \eqref {eq:integ2} holds, whereas Problem \ref {prob:gral} asks to prove much more, i.e., that $X_i=0$, for $1\leq i\leq 4$. So Problem \ref {prob:gral} does not look equivalent to Problem \ref {prob:non2}, and it is not at all clear if it has a solution or not. 
 
\section{Algebraic proofs via the Hodge-Fr\"olicher spectral sequence}\label{sec:illusie} 

This section is devoted to give a brief account on how algebraic proofs of the closedness of regular 1--forms and of the equality between algebraic and analytic irregularity (both in characteristic zero) can be obtained using modern tools.

\subsection{Global regular 1-forms are closed in characteristic zero}

Let $X$ be a smooth, irreducible and projective variety of arbitrary dimension over an algebraically closed field $\bbK$.
Let $\Omega_X^i$ be the sheaf of algebraic differential $i$-forms on $X$. The exterior derivative 
$d: \Omega_X^i \to \Omega_X^{i+1}$ allows to define a complex (the so-called \emph{algebraic 
de Rham complex}) and a spectral sequence (the so-called \emph{Hodge-Fr\"olicher spectral sequence}): 
$$
E_1 = \bigoplus_{i,j \ge 0} E_1^{i,j}
$$
where
$$
E_1^{i,j} := H^j(X, \Omega_X^i)
$$
and 
$$
d_1: E_1^{i,j} \to E_1^{i+1,j}
$$
is given by
$$
d: H^j(X, \Omega_X^i) \to H^j(X, \Omega_X^{i+1}).
$$
If this spectral sequence degenerates at $E_1$, then in particular we have $E_2^{1,0} = E_1^{1,0}$, i.e. 
$$
\frac{\mathrm{Ker} ( H^0(X, \Omega_X^1) \to H^0(X, \Omega_X^2) )}{\mathrm{Im} ( H^0(X, \mathcal{O}_X) \to H^0(X, \Omega_X^1) )}  
= \mathrm{Ker} ( H^0(X, \Omega_X^1) \to H^0(X, \Omega_X^2) ) = H^0(X, \Omega_X^1).
$$
Hence we see that if the Hodge-Fr\"olicher spectral sequence degenerates at $E_1$, then all global regular 1-forms are closed. 

An algebraic proof of the degeneration at $E_1$ of the Hodge-Fr\"olicher spectral sequence in characteristic zero has been obtained by Deligne and Illusie 
in the paper \cite{DI} published in 1987 (see also \cite{EV} and \cite{I} for more detailed and self-contained expositions). The strategy 
involves two steps: first, the result is proven under suitable assumptions in positive characteristic; then, by applying standard "spreading out" techniques, it is extended to characteristic zero. 

\begin{theorem} \emph{(\cite{I}, Corollary 5.6)} Let $k$ be a perfect field of characteristic $p$ and let $X$ be a smooth and proper 
$k$-scheme of dimension $<p$. If $X$ satisfies a technical assumption (namely, $X$ can be lifted over the ring $W_2(k)$ of Witt vectors 
of length $2$ over $k$), then the Hodge-Fr\"olicher spectral sequence of $X$ over $k$ degenerates at $E_1$.  
\end{theorem}

\begin{corollary} \emph{(\cite{I}, Theorem 6.9)} Let $\bbK$ be a field of characteristic zero and let $X$ be a smooth and proper 
$\bbK$-scheme of arbitrary dimension. Then the Hodge-Fr\"olicher spectral sequence of $X$ over $\bbK$ degenerates at $E_1$.  
\end{corollary}

For a friendly introduction to this circle of ideas we refer the interested reader to the informal survey \cite{M} (see also \cite {Mumf}, which explains the role of Witt vectors in studying the irregularity in positive characteristic). Unluckily, 
it seems that in order to address the case of surfaces one needs to apply the whole machinery developed for the general case.  
 
\subsection{Analytic irregularity and arithmetic irregularity coincide}
Let $X$ be a smooth and projective surface over the complex field $\mathbb{C}$. 
As already realized (at least implicitly) by Castelnuovo, the fact that the \emph{analytic irregularity} $q_{\rm an}(X)=h^0(X,\Omega^1_X)$ 
is equal to the \emph{arithmetic irregularity} $q_a(X)=h^1(X, \calO_X)$ (which holds in general only in characteristic zero) is strictly related to the closedness of global regular 1-forms. 

A crucial additional ingredient for proving algebraically that $q_{\rm an}(X)=q_a(X)$ is the following equality, which admits a purely algebraic proof 
(see for instance \cite[Mumford's remarks i) and iii) on p. 200]{ Za},  and \cite[Theorem 5.1] {B}): 
\begin{equation}\label{betti}
h^1(X, \mathbb{C}) = 2 h^1(X, \mathcal{O}_X) = 2 q_a(X).
\end{equation}

As in \cite{BHPV}, proof of Lemma (2.6) on p. 139, there is a natural exact sequence
$$
0 \to \mathbb{C} \to \mathcal{O}_X \to \mathcal{S} \to 0,
$$
where $\mathcal{S}$ denotes the sheaf of closed regular 1-forms on $X$. Since all global regular 1-forms are closed, we get an exact sequence
$$
0 \to H^0(X, \Omega_X^1) \to H^1(X, \mathbb{C}) \to H^1(X, \mathcal{O}_X). 
$$
It follows that $h^1(X, \mathbb{C}) \le h^0(X, \Omega_X^1) + h^1(X, \mathcal{O}_X)$ and together with (\ref{betti}) we may deduce
$$
h^1(X, \mathcal{O}_X) \le h^0(X, \Omega_X^1).
$$

On the other hand, the opposite inequality turns out to be much subtler and seems to require the full strength of the Hodge-Fr\"olicher spectral sequence. 
Indeed, if one defines the \emph{algebraic de Rham cohomology} $H^*_{\rm dR}(X/ \bbK)$ as the hypercohomology of the algebraic de Rham complex, then the 
equality
$$
\dim (H^1_{\rm dR}(X/ \bbK)) = q_{\rm an}(X) + q_a(X)
$$
is a formal consequence of the degeneration at $E_1$ of the Hodge-Fr\"olicher spectral sequence (see for instance \cite{M}, Lemma 3.4). 
In particular, for $\bbK = \mathbb{C}$ we have 
$$
q_{\rm an}(X) + q_a(X) = \dim (H^1_{\rm dR}(X/ \bbK)) = h^1(X, \mathbb{C}) = 2 q_a(X)
$$
by (\ref{betti}), hence we obtain $q_{\rm an}(X)=q_a(X)$.

\end{document}